%% file: article_GPretsel.tex
\documentclass[12pt]{amsart}
\usepackage{tikz}
\usepackage{amsmath,amssymb}
\usepackage{amsopn}
\usepackage{mathrsfs}
\usepackage{bm}
\usepackage{url}
\usepackage{listings,jvlisting}
\usepackage{tcolorbox}
\usepackage{ascmac}
\usepackage{amsthm}
\usepackage{ulem}
\usepackage{graphics}
\usepackage[export]{adjustbox}
\usepackage{tikz}
\usepackage{subcaption}
\usepackage[margin=25truemm,bottom=30truemm,top=30truemm]{geometry}
\usepackage{comment}
\usepackage{hyperref}
\usepackage{xcolor}
\usepackage{mathdots}
\usepackage{tikz-3dplot}
\usepackage[title]{appendix}

\usetikzlibrary{knots}
\usetikzlibrary{braids}
\usetikzlibrary{decorations.markings}
\usetikzlibrary{hobby}
\usetikzlibrary{arrows.meta}
\usetikzlibrary{positioning}
\usetikzlibrary{calc}

\theoremstyle{definition}
\newtheorem{definition}{Definition}[section]

\theoremstyle{plain}
\newtheorem{theorem}[definition]{Theorem}

\newtheorem{proposition}[definition]{Proposition}
\newtheorem{example}[definition]{Example}
\newtheorem{lemma}[definition]{Lemma}

\theoremstyle{remark}
\newtheorem{remark}[definition]{Remark}

\numberwithin{equation}{section}

\newcommand{\RR}{\mathbb{R}}
\newcommand{\ZZ}{\mathbb{Z}}

\newcommand{\Db}{\overline{D}}

\newcommand{\ie}{\emph{i.e.}\@ifnextchar.{\!\@gobble}{}}
\newcommand{\eg}{\emph{e.g.}\@ifnextchar.{\!\@gobble}{}}
\newcommand{\etc}{etc\@ifnextchar.{}{.\@}}

\newcommand{\la}{\left\langle}
\newcommand{\ra}{\right\rangle}

\input{knotdiagrams.tex}

\allowdisplaybreaks[1]

\title{A Generalization of Pretzel Links via Spatial Graphs}
\author{Kotaro Shoji}
\date{}
\address{Graduate School of Science and Engineering, Ehime University, 2-5 Bunkyo-cho, Matsuyama, Ehime 790-8577, Japan}
\email{m807023a@mails.cc.ehime-u.ac.jp}

\begin{document}

\begin{abstract}
In this paper, we introduce \textit{graph-pretzel links},
a generalization of classical pretzel links based on spatial graph projections. 
As our main result, 
we investigate a subfamily associated with the complete graph on four vertices to construct an infinite family of distinct ribbon knots. 
Furthermore, although they all share a trivial Alexander polynomial, 
they can be distinguished from one another by their Jones polynomials.
\end{abstract}

\maketitle

\section{Introduction}

In knot theory, classes of links such as torus links and pretzel links have long played important roles as test classes for evaluating the strength of invariants.
Torus links are obtained by simply twisting and connecting multiple strands, while pretzel links are formed by connecting multiple 2-strand twists side by side. 
Generalizing these constructions to twists of three or more strands connected in a pretzel-like fashion seems to be a highly natural extension, 
yet it has not been extensively studied. Motivated by this observation, we introduce a new class of links that naturally generalizes both torus and pretzel links.

Let $\Gamma$ be a projection of a spatial graph with vertices $v_1, v_2, \dots, v_i$. 
We construct a link by taking $\Gamma$ and its mirror image, truncating them at their vertices, 
and connecting the corresponding endpoints at $v_j$ with a certain number of twists specified by an integer $n_j$. 
We denote the resulting link by $P(\Gamma, \{n_1, \dots, n_i\})$ and call it a \textbf{graph-pretzel link}. 
The rigorous definition of this construction is given in Section 2 (Definition \ref{maindef}).

To investigate the properties of this newly constructed class, 
we first focus on polynomial invariants. 
Polynomial invariants are relatively easy to compute while possessing strong distinguishing power. 
However, they are not complete invariants, and pairs or infinite families of distinct knots sharing the same polynomial invariants have been widely studied for a long time. 
For example, mutant knots share the same HOMFLYPT and Kauffman polynomials. 
Birman \cite{MR799267} provided pairs of knots obtained as closures of 3-braids that share the same Alexander and Jones polynomials. 
As for infinite families, Kauffman and Lopes \cite{MR3687484} constructed infinite families of distinct knots with the same Alexander polynomial (distinguishable by the Jones polynomial) using pretzel knots, 
and Kanenobu \cite{MR831406,kanenobutwo} constructed infinite families sharing the same HOMFLYPT polynomial. 
Using the graph-pretzel links defined in this paper, 
we obtain the following infinite family of knots.

\begin{theorem}\label{maintheorem}
    Let $K_{n}:=P(\adjustbox{valign=c}{\scalebox{0.4}{\tetragraphwithv}},\{-2,2,-(3n+1),3\})$.
    Then, we have
    \begin{equation*}
        \Delta_{K_{n}}(t) = 1.
    \end{equation*}
    Furthermore, $J_{K_n}(q)\neq J_{K_m}(q)$ if $n\neq m$. 
    In particular, $K_n$ is not isotopic to $K_m$. 
\end{theorem}

Note that $K_{n}$ is a knot (a one-component link). 
The fact that the Alexander polynomial not only fails to distinguish the knots in this family, 
but is also identically trivial ($\Delta_{K_n}(t)=1$), 
has profound significance in low-dimensional topology.
Knots with trivial Alexander polynomial are algebraically indistinguishable from the unknot and are known to be topologically slice by Freedman \cite{10.4310/jdg/1214437136}. 
On the other hand, whether they are all smoothly slice remains an intriguing question. 
For instance, Fintushel and Stern \cite{e6fd0a6e-64c3-3748-882e-9b5c2b5bf060} showed that $3$-pretzel knots with all odd parameters and trivial Alexander polynomial are not smoothly slice. 
In contrast, since every ribbon knot is smoothly slice, the family of graph-pretzel knots $K_n$ constructed in this paper is smoothly slice in a stronger sense, 
as shown in the following theorem.

\begin{theorem}\label{maintheoremtwo}
    For any positive integer $n$, $K_n$ is a ribbon knot. 
    In particular, $K_n$ is smoothly slice.
\end{theorem}

\begin{remark}
    $K_{0}$ is the unknot. 
    $K_{1}$ is identified by its DT name as $K14n22185$ (by SnapPy \cite{SnapPy}).
    While we establish that $K_{1}$ is a ribbon knot with $\Delta(t)=1$, 
    it is also known to be a hyperbolic knot whose $0$-surgery yields a toroidal manifold (see \cite{abe2026complexityequal0surgeries})
    even though the genus of $K_1$ is two, rather than one.
    Therefore, $K_1$ provides a rare example combining these properties.
    This suggests that our framework of graph-pretzel links may serve as a promising search space for links with unusual geometric and topological properties.
\end{remark}

This paper is organized as follows: 
Section 2 introduces graph-pretzel links. Section 3 examines the complete graph on four vertices. 
In Section 4, we provide the proofs of the main results.

\section{Definition and Properties of the link class}
 
In this section, we give a rigorous definition of graph-pretzel links, 
investigate some of their basic properties
and show that this class naturally encompasses classical torus and pretzel links.

\begin{definition}\label{maindef}
    Let $\Gamma$ be a projection of a spatial graph, and let $\{v_1, v_2, \ldots, v_i\}$ be the set of vertices of $\Gamma$. 
    Let $r_j$ be the valence of each vertex $v_j$. 
    Let $\overline{\Gamma}$ be the mirror image of $\Gamma$, and let $\{\overline{v_1}, \overline{v_2}, \ldots, \overline{v_i}\}$ be the set of vertices of $\overline{\Gamma}$, where each $\overline{v_j}$ corresponds to $v_j$. 
    Then:

    \begin{enumerate}
        \item In $\mathbb{R}^3$, we embed $\Gamma$ as a spatial graph
        into the region $\{(x,y,z)\mid\frac{9}{10}\leq z\leq\frac{11}{10}\}$, whose projection onto $\RR^2\times\{1\}$ is $\Gamma$ itself. 
        We embed $\overline{\Gamma}$ into the region $\{(x,y,z)\mid-\frac{11}{10}\leq z\leq-\frac{9}{10}\}$ whose projection onto $\mathbb{R}^2 \times \{-1\}$ is $\overline{\Gamma}$ itself,
        such that they are symmetric with respect to the plane $\mathbb{R}^2 \times \{0\}$, where each $v_j$ is mapped to $\overline{v_j}$ by the reflection.
        \item Truncate $\Gamma$ and $\overline{\Gamma}$ at each of their vertices.
        \item Let $n_1, n_2, \ldots, n_i$ be integers. For each vertex $v_j$, connect the $r_j$ endpoints of the truncated $\Gamma$ at $v_j$ to the $r_j$ endpoints of the truncated $\overline{\Gamma}$ at $\overline{v_j}$ with $r_j$ strands, 
        applying an $n_j/r_j$ twist.
    \end{enumerate}

    The (unoriented) link consisting of the truncated $\Gamma$, $\overline{\Gamma}$, and the connecting strands is denoted by $P(\Gamma, \{n_1, n_2, \ldots, n_i\})$ and is called a \textbf{$\Gamma$-pretzel link}. A link constructed in this manner is called a \textbf{graph-pretzel link}.
\end{definition}

We refer to $\Gamma$-pretzel links and graph-pretzel links as $\Gamma$-pretzel knots and graph-pretzel knots, respectively, if they are knots.

\begin{remark}\label{fulltwist}
Let $n$ be an integer and $r$ be a positive integer.
An $n/r$ twist is formed by stacking $n$ copies of a $1/r$ twist (Figure \ref{twistsylinder}) for $n\geq0$, $-n$ copies of the mirror image of a $1/r$ twist for $n<0$. 
Furthermore, an $r/r$ twist is referred to as a full twist.
\end{remark}

\begin{figure}[htbp]
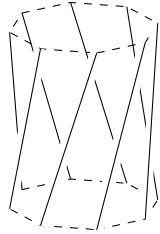

    \centering
    \twistsylinder
    \caption{1/8 twist}
    \label{twistsylinder}
\end{figure}

The next two examples demonstrate that the class of graph-pretzel links encompasses torus links and pretzel links as subclasses.

\begin{example}
    Let $\Gamma_i$ be the graph shown in Figure \ref{fig:d2-2}.
    Then, $P(\Gamma_i,\{n_1,n_2\})$ with a suitable orientation is isotopic to $T(i,n_1+n_2)$.
\end{example}

\begin{example}
    Let $\Gamma_i$ be an $i$-gon. Then, $P(\Gamma_i,\{n_1,\ldots, n_i\})$ is isotopic to (classical) pretzel knots 
    $P(n_1,n_2,\ldots,n_i)$.
\end{example}

\begin{figure}
    \centering
    \twopointgraph
    \caption{}
    \label{fig:d2-2}
\end{figure}

By definition, we have the following proposition about the bridge index of 
graph-pretzel knots.

\begin{proposition}
    Let $K = P(\Gamma,\{n_1,\ldots, n_i\})$.
    The bridge index of $K$ is less than or equal to
    the number of edges in $\Gamma$.
\end{proposition}

Also by definition, 
the mirror image of a graph-pretzel link is given by the following proposition.

\begin{proposition}
    The mirror image of $P(\Gamma,\{n_1,\ldots,n_i\})$ is isotopic to $P(\Gamma, \{-n_1,\ldots,-n_i\})$.
\end{proposition}

\section{The case for the complete graph on four vertices}

In this section, we discuss the link $P(\Gamma,\{n_1,n_2,n_3,n_4\})$
for the case where $\Gamma$ is the complete graph on four vertices.
This family of links is used in Theorem \ref{maintheorem}.

In what follows, 
we assume $\Gamma=\adjustbox{valign=c}{\scalebox{0.4}{\tetragraphwithv}}$.
From the symmetry of the tetrahedron, we obtain the following proposition.
\begin{proposition}\label{mainprop}
    $P(\Gamma,\{n_1,n_2,n_3,n_4\})$ is invariant under the action of the symmetric group $S_4$ 
    on the parameters $(n_1,n_2,n_3,n_4)$.
\end{proposition}
\begin{proof}
    The link $P(\Gamma,\{n_1,n_2,n_3,n_4\})$ can be formed as shown in Figure \ref{fig:tetrahedron-b}.
    The symmetry of the tetrahedron reveals its invariance under the 
    congruence transformations of the tetrahedron.
    The tetrahedral group is isomorphic to the alternating group of degree four $A_4$.
    Thus, we confirm its invariance under $A_4$.

    Next, we consider a $\pi$ rotation of the constructed $\Gamma$-pretzel link that interchanges the top and bottom in $\mathbb{R}^3$. 
    Mapping this back onto the tetrahedron corresponds to interchanging two of the original parameters, say, $n_i$ and $n_j$. 
    This operation effectively adds a transposition of two vertices to the tetrahedral group,
    which implies that the link is invariant under the action of the symmetric group $S_4$.
\end{proof}

\begin{figure}
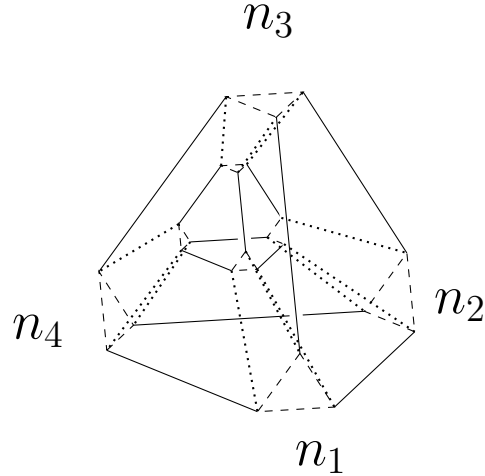

    \centering
    \truncatedtetraknott
    \caption{Broken lines represent twist regions. $n_i$ shows the number of twists.}
    \label{fig:tetrahedron-b}
\end{figure}

\begin{remark}
    $P(\Gamma,\{n_1,n_2,n_3,n_4\})$ has a pretzel-like diagram. 
    See Figure \ref{fig}.
\end{remark}

\begin{figure}[htbp]
    \centering
    \begin{subfigure}{0.45\linewidth}
        \centering
        \scalebox{0.5}{\threepretzeln}
        \caption{From left to right: $n_1,n_2,n_3,n_4$ ($n_i$ is a positive integer)}
        \label{fig:d2-3}
    \end{subfigure}
    \hspace{0.05\linewidth}
    \begin{subfigure}{0.45\linewidth}
        \centering
        \scalebox{0.5}{\threepretzelmn}
        \caption{From left to right: $-n_1, -n_2, -n_3, -n_4$ $(n_i\text{ is a positive integer})$}
        \label{fig:d2-4}
    \end{subfigure}
    \caption{}
    \label{fig}
\end{figure}

\begin{remark}
    By Proposition \ref{mainprop}, 
    the knot $K_n$ defined in Theorem 1.1 does not depend on the order of the parameters $\{-2, 2, -(3n+1), 3\}$.
\end{remark}

\section{Proof of Main Theorems}

In this section, we prove two main theorems. 

\subsection{Proof of Theorem \ref{maintheorem}}
First, we compute the Alexander polynomial using the skein relation and 
the Jones polynomial using the Kauffman bracket (for details, see, for example \cite{Lickorish}). 

For computational convenience, 
we use the Conway polynomial instead of the Alexander polynomial. 
The Conway polynomial $\nabla(L)\in \ZZ[z]$ for a link $L$ 
is defined such that its value for the unknot is $1$,
and it satisfies the skein relation:
\begin{equation*}
    \nabla_{\positivecrossing}(z) - \nabla_{\negativecrossing}(z) = z\nabla_{\smoothing}(z).
\end{equation*}
The Alexander polynomial of a link $L$, denoted $\Delta_{L}(t)$, can be recovered from the
Conway polynomial by substituting
$z = t^{\frac{1}{2}}-t^{-\frac{1}{2}}$.

The Kauffman bracket $\la L\ra\in \ZZ[A,A^{-1}]$ for an unoriented link $L$ 
is defined such that its value for the unknot with no crossings is 1, 
and it satisfies the relation:
\begin{equation*}
    \left\langle\crossing\right\rangle = A\left\langle\crossingzero\right\rangle + A^{-1}\left\langle\crossinginfty\right\rangle.
\end{equation*}
Using the Kauffman bracket, the Jones polynomial of a link $L$ is given by
\begin{equation*}
    J_L(q)= (-A^3)^{-w(D)}\langle\Db\rangle|_{A^2=q^{-1/2}},
\end{equation*}
where $D$ is a diagram of $L$, $\overline{D}$ is an unoriented diagram obtained by forgetting the orientation of $D$, and $w(D)$ is the writhe of $D$.

To prove Theorem \ref{maintheorem}, we show the following lemma.
\begin{lemma}\label{mainlemma}
    \begin{equation*}
        \begin{split}
            J_{K_{n}}(q) &= (q^{3n+2}-q^{2})(1-q^{-1}+q^{-3}-2q^{-4}+q^{-5}-q^{-7}+q^{-8})+1.
        \end{split}
    \end{equation*}
\end{lemma}

\begin{proof}
    Applying the Kauffman bracket relation to the full twists in Figure \ref{todos}, 
    we have a recurrence relation
    \begin{equation*}
        \left\langle \overline{K_{n}} \right\rangle-(A^6+A^{-6})\left\langle \overline{K_{n-1}}\right\rangle + \left\langle \overline{K_{n-2}} \right\rangle=0.
    \end{equation*}
    We have $\la \overline{K_0}\ra = 1$, and the initial value
    \begin{equation*}
        \left\langle \overline{K_{1}} \right\rangle=-A^{24}+A^{20}+A^8-A^4+2-A^{-4}-A^{-16} + A^{-20}
    \end{equation*}
    can be verified directly by the Kauffman bracket relation.
    Solving this recurrence relation with the initial values, we obtain 
    \begin{equation*}
        \la \overline{K_{n}} \ra = -A^{-14}(A^{6n}-A^{-6n})(1-A^4+A^{12}-2A^{16}+A^{20}-A^{28}+A^{32})+A^{6(n-1)}.
    \end{equation*}

    We obtain $w(K_n)=2n-2$ from Figure \ref{todoe}.

    Now we can compute the Jones polynomial of $K_n$ as
    \begin{alignat*}{1}
        J_{K_{n}}(q) &= -(A^{-6(n-1)})A^{-14}(A^{6n}-A^{-6n})(1-A^4+A^{12}-2A^{16}+A^{20}-A^{28}+A^{32})+1\\
        &=(A^{-12n-8}-A^{-8})(1-A^4+A^{12}-2A^{16}+A^{20}-A^{28}+A^{32})+1 \\
        &=(q^{3n+2}-q^{2})(1-q^{-1}+q^{-3}-2q^{-4}+q^{-5}-q^{-7}+q^{-8})+1.
    \end{alignat*}
\end{proof}

\begin{figure}[htbp]
    \centering
    \begin{subfigure}{0.45\linewidth}
        \centering
        \scalebox{0.3}{\includegraphics{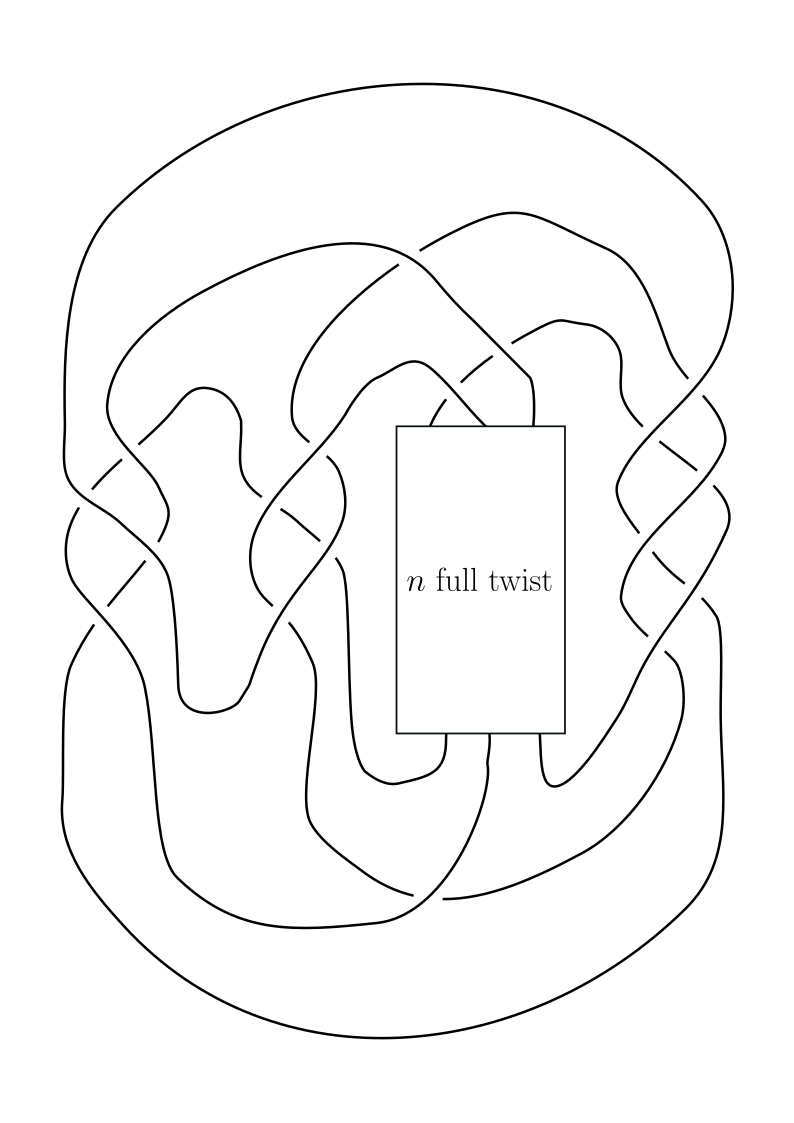}}
        \caption{}
        \label{todos}
    \end{subfigure}
    \hspace{0.05\linewidth}
    \begin{subfigure}{0.45\linewidth}
        \centering
        \scalebox{0.3}{\includegraphics{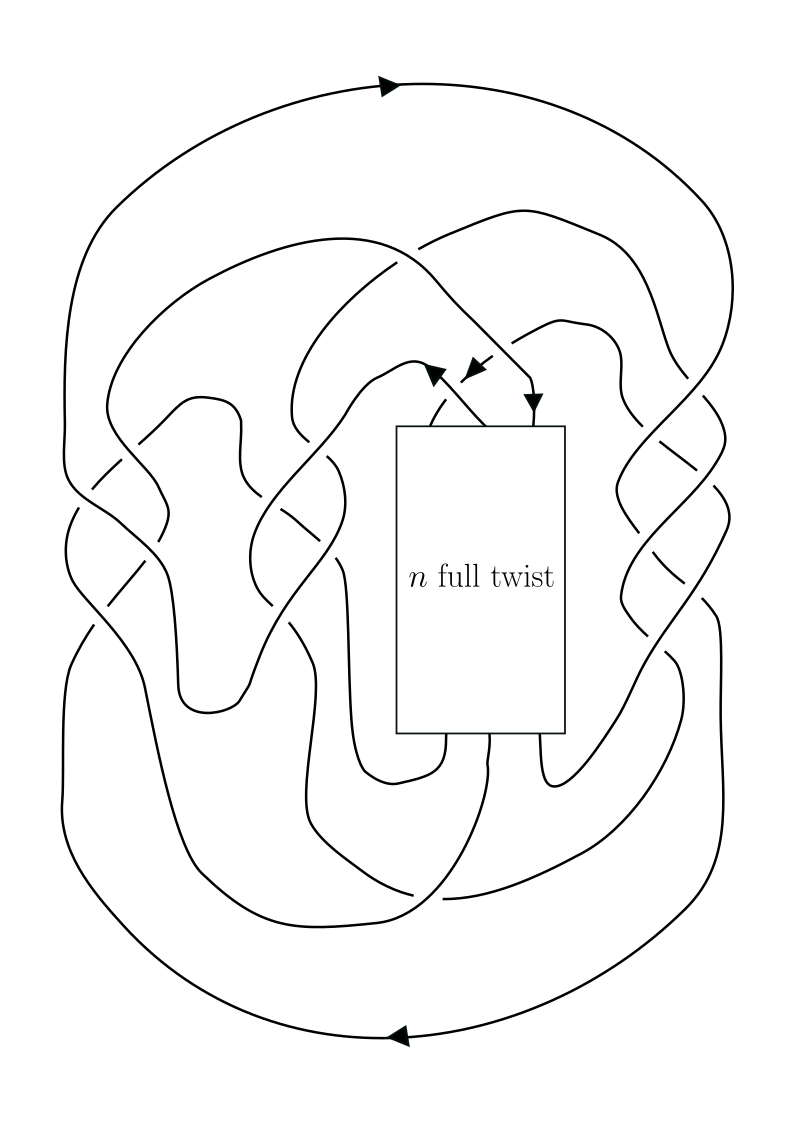}}
        \caption{}
        \label{todoe}
    \end{subfigure}
    \caption{$n$ in the diagram is a negative integer}
    \label{figer}
\end{figure}

Now we can prove Theorem \ref{maintheorem}.
\begin{proof}[Proof of Theorem \ref{maintheorem}]
    Applying the skein relation of the Conway polynomial 
    to the full twists in Figure \ref{todoe}, we have the recurrence relation
    \begin{equation*}
        \nabla_{K_{n}}(z) -(3+z^2)\nabla_{K_{n-1}}(z)+(3+z^2)\nabla_{K_{n-2}}(z)-\nabla_{K_{n-3}}(z) =0.
    \end{equation*}
    Since $K_0$ is the unknot, we have $\nabla_{K_0}(z)=1$. 
    The initial values $\nabla_{K_1}(z)=\nabla_{K_2}(z)=1$ 
    can be obtained by the skein relation of the Conway polynomial.
    Consequently, we get $\nabla_{K_n}(z)=1$ and so $\Delta_{K_n}(t)=1$.
    From Lemma \ref{mainlemma}, we have
    \begin{equation*}
        \begin{split}
            J_{K_{n}}(q) &= (q^{3n+2}-q^{2})(1-q^{-1}+q^{-3}-2q^{-4}+q^{-5}-q^{-7}+q^{-8})+1.
        \end{split}
    \end{equation*}
    This implies that $K_n$ and $K_m$ are not isotopic if $n\neq m$.
\end{proof}

\subsection{Proof of Theorem \ref{maintheoremtwo}}

\begin{proof}[Proof of Theorem \ref{maintheoremtwo}]
By cutting the band encircled by the red dashed line in Figure \ref{zu}, 
we obtain two trivial circles, as shown in Figure \ref{zuni}. 
This shows that $K_1$ is a band sum of two trivial circles, 
which implies that it is a ribbon knot. 
Moreover, adding negative full twists does not affect this property. 
Therefore, we can conclude that $K_n$ is also a ribbon knot for any positive integer $n$. 
This concludes the proof.
\end{proof}

\begin{figure}[htbp]
    \centering
    \begin{subfigure}{0.45\linewidth}
        \centering
        \scalebox{0.3}{\includegraphics{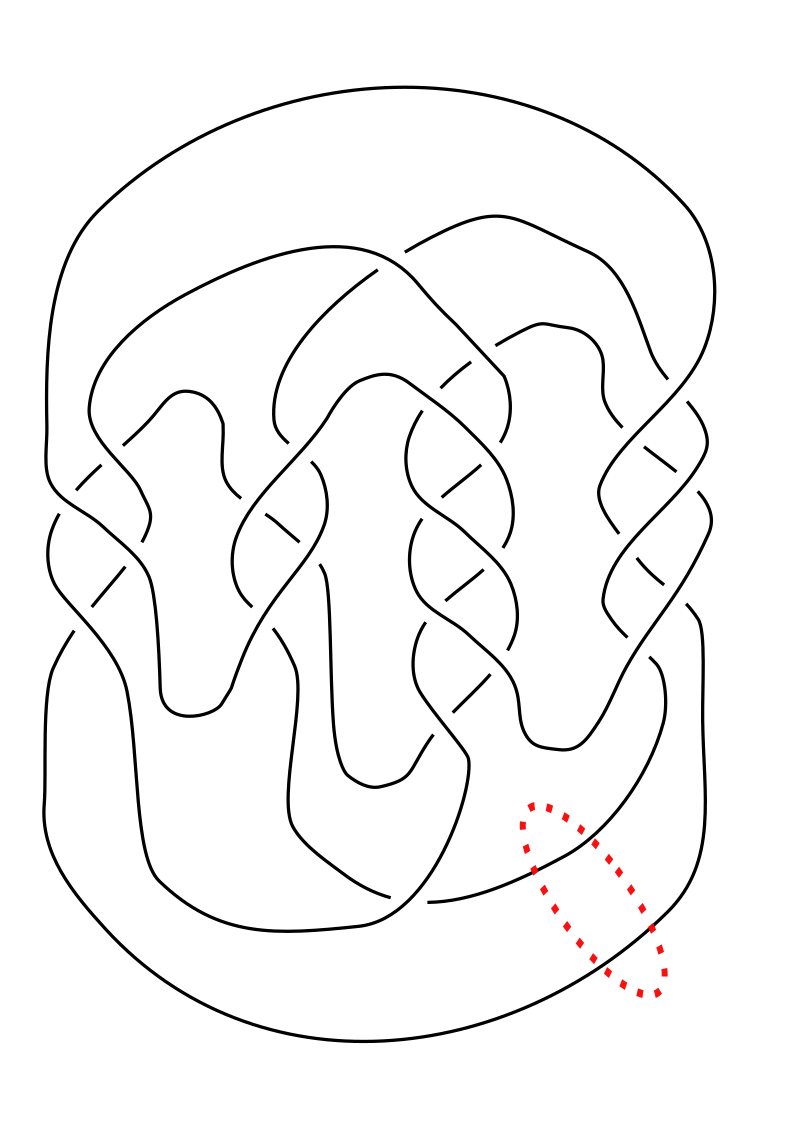}}
        \caption{$K_1$}
        \label{zu}
    \end{subfigure}
    \hspace{0.05\linewidth}
    \begin{subfigure}{0.45\linewidth}
        \centering
        \scalebox{0.3}{\includegraphics{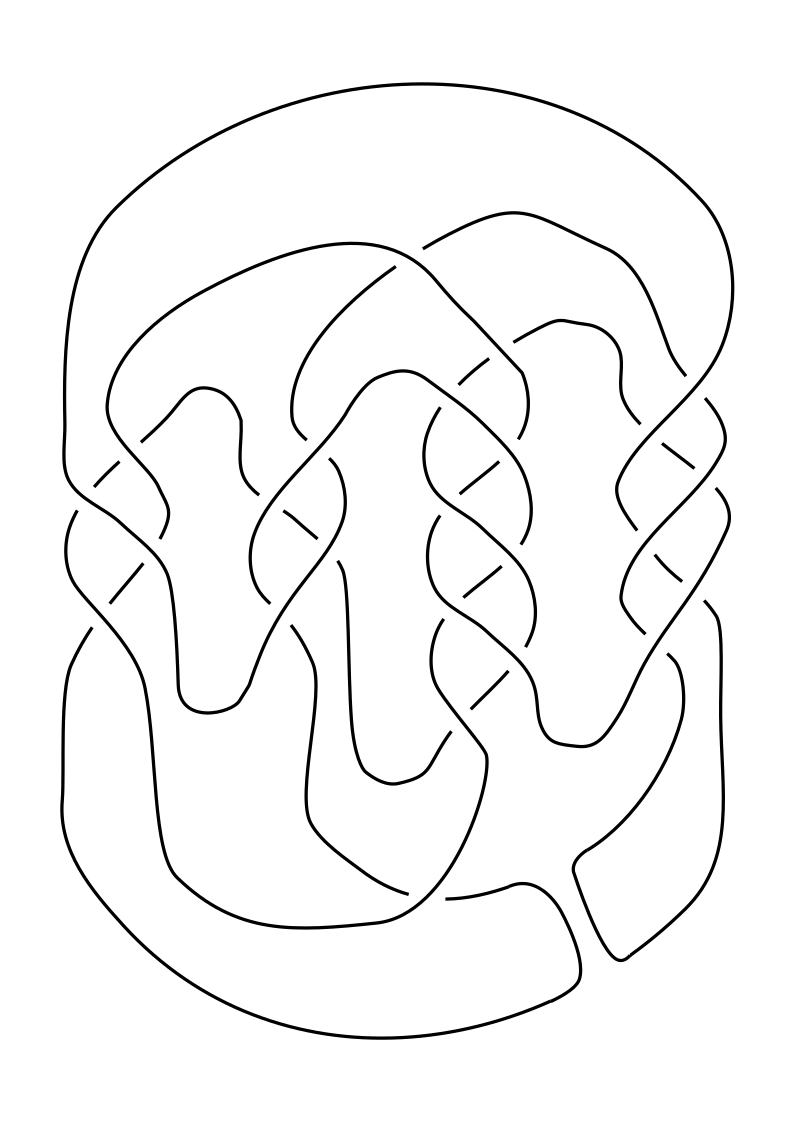}}
        \caption{}
        \label{zuni}
    \end{subfigure}
    \caption{Band surgery}
    \label{nahan}
\end{figure}

\section{Concluding Remarks}
In this paper, we introduced the framework of graph-pretzel links and demonstrated its utility by constructing an infinite family of ribbon knots with $\Delta(t)=1$ from the complete graph on four vertices. 
A natural direction for future research is to investigate graph-pretzel links associated with other spatial graphs. 
Exploring how the choice of the underlying graph and twist parameters relates to various knot invariants may yield further interesting examples and infinite families in low-dimensional topology.

\section*{Acknowledgements}
The author would like to thank Shin-ichi Oguni for his thoughtful guidance and helpful discussions about this work.

\bibliographystyle{plain}
\bibliography{sanko}

\end{document}

%% file: knotdiagrams.tex
\newcommand{\crossing}{
    \begin{tikzpicture}[scale=0.25,baseline = {(0,-0.1)}]
        \draw (-1,1) to (1,-1);
        \draw[white, line width = 7pt] (0.5,0.5) to (-0.5,-0.5);
        \draw (1,1) to (-1,-1);
    \end{tikzpicture}
}

\newcommand{\crossingzero}{
    \begin{tikzpicture}[scale=0.25,baseline = {(0,-0.1)}]
        \draw (1, 1) to[curve through = {(0.5, 0)}] (1, -1);
        \draw (-1, 1) to[curve through = {(-0.5, 0)}] (-1, -1);
    \end{tikzpicture}
}

\newcommand{\crossinginfty}{
    \begin{tikzpicture}[scale=0.25,baseline = {(0,-0.1)}]
        \draw (1, 1) to[curve through = {(0, 0.5)}] (-1, 1);
        \draw (1, -1) to[curve through = {(0, -0.5)}] (-1, -1);
    \end{tikzpicture}
}

\newcommand{\positivecrossing}{
    \begin{tikzpicture}[scale=0.2,baseline = {(0,-0.1)}]
        \draw[->] (-1,1) to (1,-1);
        \draw[white, line width = 7pt] (0.5,0.5) to (-0.5,-0.5);
        \draw[->] (1,1) to (-1,-1);
    \end{tikzpicture}
}

\newcommand{\negativecrossing}{
    \begin{tikzpicture}[scale=0.2,baseline = {(0,-0.1)}]
        \draw[->] (1,1) to (-1,-1);
        \draw[white, line width = 7pt] (-0.5,0.5) to (0.5,-0.5);
        \draw[->] (-1,1) to (1,-1);
    \end{tikzpicture}
}

\newcommand{\smoothing}{
    \begin{tikzpicture}[scale=0.25,baseline = {(0,-0.1)}]
        \draw[->] (1, 1) to[curve through = {(0.5, 0)}] (1, -1);
        \draw[->] (-1, 1) to[curve through = {(-0.5, 0)}] (-1, -1);
    \end{tikzpicture}
}

\newcommand{\twistsylinder}{
    \tdplotsetmaincoords{70}{80} 
    \begin{tikzpicture}[scale=0.5]
        \begin{scope}[tdplot_main_coords]
            \def\radius{2}
            \def\height{5}
            \def\numpts{8} % Number of points on each circle

            \foreach \i in {0,...,\numpts} {
                \pgfmathsetmacro{\angle}{(\i/\numpts)*360}
                \coordinate (U\i) at (\radius*cos{\angle}, \radius*sin{\angle}, \height);
            }
            \foreach \i in {0,...,\numpts} {
                \pgfmathsetmacro{\angle}{(\i/\numpts)*360}
                \coordinate (L\i) at (\radius*cos{\angle}, \radius*sin{\angle}, 0);
            }

            \draw[dashed] (L0) \foreach \i in {1,...,\numpts} {-- (L\i)} -- cycle; % Bottom circle
            \draw[dashed] (U0) \foreach \i in {1,...,\numpts} {-- (U\i)} -- cycle; % Top circle
            
            \foreach \i in {0,...,\numpts} {
                \pgfmathsetmacro{\nextval}{int(mod(\i, 8)) + 1}
                \pgfmathsetmacro{\isTarget}{ifthenelse(\i > 5 || \i ==  0, 1, 0)}
                \ifnum \isTarget =1
                    \draw[white,line width=5pt] (U\nextval) -- (L\i);
                \fi
                \draw (U\nextval) -- (L\i);
            }
            \draw[white,line width=4pt] (U2) -- (L1);
            \draw (U2) -- (L1);
        \end{scope}
    \end{tikzpicture}
}

\newcommand{\truncatedtetraknott}{
    \tdplotsetmaincoords{70}{80} 
    \begin{tikzpicture}[scale=0.5]
        \begin{scope}[tdplot_main_coords] 
            \coordinate (A) at (2, 0, 0);
            \coordinate (B) at (-1, 1.732, 0);
            \coordinate (C) at (-1, -1.732, 0);
            \coordinate (D) at (0, 0, 3); 

            \coordinate (A2) at (6, 0, 0);
            \coordinate (B2) at (-3, 1.732*3, 0);
            \coordinate (C2) at (-3, -1.732*3, 0);
            \coordinate (D2) at (0, 0, 9);

            \pgfmathsetmacro{\truncfrac}{0.20} % Adjust this value as needed

            \coordinate (AB1) at ($(A)!\truncfrac!(B)+(0,-0.9,3)$);
            \coordinate (BA1) at ($(B)!\truncfrac!(A)+(0,-0.9,3)$);
            \coordinate (BD1) at ($(B)!\truncfrac!(D)+(0,-0.9,3)$);
            \coordinate (DB1) at ($(D)!\truncfrac!(B)+(0,-0.9,3)$);
            \coordinate (DA1) at ($(D)!\truncfrac!(A)+(0,-0.9,3)$);
            \coordinate (AD1) at ($(A)!\truncfrac!(D)+(0,-0.9,3)$);
            \coordinate (AC1) at ($(A)!\truncfrac!(C)+(0,-0.9,3)$);
            \coordinate (CA1) at ($(C)!\truncfrac!(A)+(0,-0.9,3)$);
            \coordinate (CD1) at ($(C)!\truncfrac!(D)+(0,-0.9,3)$);
            \coordinate (DC1) at ($(D)!\truncfrac!(C)+(0,-0.9,3)$);
            \coordinate (BC1) at ($(B)!\truncfrac!(C)+(0,-0.9,3)$);
            \coordinate (CB1) at ($(C)!\truncfrac!(B)+(0,-0.9,3)$);

            \coordinate (A2B2_1) at ($(A2)!\truncfrac!(B2)$);
            \coordinate (B2A2_1) at ($(B2)!\truncfrac!(A2)$);
            \coordinate (B2D2_1) at ($(B2)!\truncfrac!(D2)$);
            \coordinate (D2B2_1) at ($(D2)!\truncfrac!(B2)$);
            \coordinate (D2A2_1) at ($(D2)!\truncfrac!(A2)$);
            \coordinate (A2D2_1) at ($(A2)!\truncfrac!(D2)$);
            \coordinate (A2C2_1) at ($(A2)!\truncfrac!(C2)$);
            \coordinate (C2A2_1) at ($(C2)!\truncfrac!(A2)$);
            \coordinate (C2D2_1) at ($(C2)!\truncfrac!(D2)$);
            \coordinate (D2C2_1) at ($(D2)!\truncfrac!(C2)$);
            \coordinate (B2C2_1) at ($(B2)!\truncfrac!(C2)$);
            \coordinate (C2B2_1) at ($(C2)!\truncfrac!(B2)$);

            \draw (A2B2_1) -- (B2A2_1);
            \draw (B2D2_1) -- (D2B2_1);
            \draw (A2C2_1) -- (C2A2_1);
            \draw (C2D2_1) -- (D2C2_1);
            \draw (B2C2_1) -- (C2B2_1);

            \draw[white,line width=5pt] (AB1) -- (BA1);
            \draw[white,line width=5pt] (BD1) -- (DB1);
            \draw[white,line width=5pt] (AC1) -- (CA1);
            \draw[white,line width=5pt] (CD1) -- (DC1);
            \draw[white,line width=5pt] (BC1) -- (CB1);

            \draw (AB1) -- (BA1);
            \draw (BD1) -- (DB1);
            \draw (AC1) -- (CA1);
            \draw (CD1) -- (DC1);
            \draw (BC1) -- (CB1);

            \draw[white,line width=5pt] (DA1) -- (AD1);
            \draw (DA1) -- (AD1);

            \draw[dashed] (AD1) -- (AC1) -- (AB1) -- cycle;
            \draw[dashed] (BA1) -- (BC1) -- (BD1) -- cycle;
            \draw[dashed] (CA1) -- (CB1) -- (CD1) -- cycle;
            \draw[dashed] (DA1) -- (DB1) -- (DC1) -- cycle;

            \draw[white,line width=5pt] (D2A2_1) -- (A2D2_1);
            \draw (D2A2_1) -- (A2D2_1);

            \draw[dashed] (A2D2_1) -- (A2C2_1) -- (A2B2_1) -- cycle;
            \draw[dashed] (B2A2_1) -- (B2C2_1) -- (B2D2_1) -- cycle;
            \draw[dashed] (C2A2_1) -- (C2B2_1) -- (C2D2_1) -- cycle;
            \draw[dashed] (D2A2_1) -- (D2B2_1) -- (D2C2_1) -- cycle;

            \draw[dotted,thick] (AB1) -- (A2B2_1);
            \draw[dotted,thick] (BA1) -- (B2A2_1);
            \draw[dotted,thick] (BD1) -- (B2D2_1);
            \draw[dotted,thick] (DB1) -- (D2B2_1);
            \draw[dotted,thick] (DA1) -- (D2A2_1);
            \draw[dotted,thick] (AD1) -- (A2D2_1);
            \draw[dotted,thick] (AC1) -- (A2C2_1);
            \draw[dotted,thick] (CA1) -- (C2A2_1);
            \draw[dotted,thick] (CD1) -- (C2D2_1);
            \draw[dotted,thick] (DC1) -- (D2C2_1);
            \draw[dotted,thick] (BC1) -- (B2C2_1);
            \draw[dotted,thick] (CB1) -- (C2B2_1);

            \node at ($(A2)+(2,0,0)$) {\LARGE $n_1$};
            \node at ($(B2)+(0,0.5,0)$) {\LARGE $n_2$};
            \node at ($(C2)+(0,-0.5,0)$) {\LARGE $n_4$};
            \node at ($(D2)+(0,0,0.5)$) {\LARGE $n_3$};

        \end{scope}
    \end{tikzpicture}
}

\newcommand{\tetragraphwithv}{
    \begin{tikzpicture}[scale=1.5,rotate=180]
        \coordinate (A) at (0,0);
        \coordinate (B) at (0,1);
        \coordinate (C) at (1,-0.5);
        \coordinate (D) at (-1,-0.5);

        \draw (A)--(B)--(C)--(D)--(A)--(C)--(D)--(B);
        \fill (A) circle (1.2pt);
        \fill (B) circle (1.2pt);
        \fill (C) circle (1.2pt);
        \fill (D) circle (1.2pt);

        \node[rotate=180,transform shape] at (0.3,0.2) {\Large $v_3$};
        \node[rotate=180,transform shape] at (0,1.25) {\Large $v_1$};
        \node[rotate=180,transform shape] at (1.3,-0.6) {\Large $v_2$};
        \node[rotate=180,transform shape] at (-1.3,-0.6) {\Large $v_4$};
    \end{tikzpicture}
}

\newcommand{\twopointgraph}{
    \begin{tikzpicture}
        \coordinate (A) at (0,0);
        \coordinate (B) at (2,0);

        \draw (A) to[curve through = {(1,0.5)}] (B);
        \draw (A) to[curve through = {(1,0.25)}] (B);
        \node at (1,-0) {$\vdots$};
        \draw (A) to[curve through = {(1,-0.5)}] (B);

        \fill (A) circle (1.2pt);
        \fill (B) circle (1.2pt);
    \end{tikzpicture}
}

\newcommand{\threepretzeln}{
    \begin{tikzpicture}
        \pic[
        braid/.cd,
        number of strands=3,
        strand 1/.style={black},
        strand 2/.style={black},
        strand 3/.style={black},
        gap=0.1,
        scale=1
        ] {braid={s_1^{-1} s_2^{-1}}};
        \node at (1,-3) {\Huge $\vdots$};
        \begin{scope}[yshift=-4cm]
            \pic[
            braid/.cd,
            number of strands=3,
            strand 1/.style={black},
            strand 2/.style={black},
            strand 3/.style={black},
            gap=0.1,
            scale=1
            ] {braid={s_1^{-1} s_2^{-1}}};
        \end{scope}

        \begin{scope}[xshift=3cm]
            \pic[
            braid/.cd,
            number of strands=3,
            strand 1/.style={black},
            strand 2/.style={black},
            strand 3/.style={black},
            gap=0.1,
            scale=1
            ] {braid={s_2^{-1} s_1^{-1}}};
            \node at (1,-3) {\Huge $\vdots$};
            \begin{scope}[yshift=-4cm]
                \pic[
                braid/.cd,
                number of strands=3,
                strand 1/.style={black},
                strand 2/.style={black},
                strand 3/.style={black},
                gap=0.1,
                scale=1
                ] {braid={s_2^{-1} s_1^{-1}}};
            \end{scope}
        \end{scope}

        \begin{scope}[xshift=6cm]
            \pic[
            braid/.cd,
            number of strands=3,
            strand 1/.style={black},
            strand 2/.style={black},
            strand 3/.style={black},
            gap=0.1,
            scale=1
            ] {braid={s_1^{-1} s_2^{-1}}};
            \node at (1,-3) {\Huge $\vdots$};
            \begin{scope}[yshift=-4cm]
                \pic[
                braid/.cd,
                number of strands=3,
                strand 1/.style={black},
                strand 2/.style={black},
                strand 3/.style={black},
                gap=0.1,
                scale=1
                ] {braid={s_1^{-1} s_2^{-1}}};
            \end{scope}
        \end{scope}

        \begin{scope}[xshift=9cm]
            \pic[
            braid/.cd,
            number of strands=3,
            strand 1/.style={black},
            strand 2/.style={black},
            strand 3/.style={black},
            gap=0.1,
            scale=1
            ] {braid={s_2^{-1} s_1^{-1}}};
            \node at (1,-3) {\Huge $\vdots$};
            \begin{scope}[yshift=-4cm]
                \pic[
                braid/.cd,
                number of strands=3,
                strand 1/.style={black},
                strand 2/.style={black},
                strand 3/.style={black},
                gap=0.1,
                scale=1
                ] {braid={s_2^{-1} s_1^{-1}}};
            \end{scope}
        \end{scope}

        \foreach \i in {2,5,8} {
            \draw (\i,0) to[in=90,out=90] (\i+1,0);
        }

        \foreach \i in {2,5,8} {
            \draw (\i,-6.5) to[in=-90,out=-90] (\i+1,-6.5);
        }

        \draw (0,0) to[in=90,out=90] (11,0);
        \draw (0,-6.5) to[in=-90,out=-90] (11,-6.5);

        \draw (4,0) to[in=90,out=90] (10,0);

        \draw[white,line width=5pt] (1,0) to[in=90,out=90] (7,0);
        \draw (1,0) to[in=90,out=90] (7,0);

        \draw (4,-6.5) to[in=-90,out=-90] (10,-6.5);

        \draw[white,line width=5pt] (1,-6.5) to[in=-90,out=-90] (7,-6.5);
        \draw (1,-6.5) to[in=-90,out=-90] (7,-6.5);

    \end{tikzpicture}
}

\newcommand{\threepretzelmn}{
    \begin{tikzpicture}
        \pic[
        braid/.cd,
        number of strands=3,
        strand 1/.style={black},
        strand 2/.style={black},
        strand 3/.style={black},
        gap=0.1,
        scale=1
        ] {braid={s_2^{1} s_1^{1}}};
        \node at (1,-3) {\Huge $\vdots$};
        \begin{scope}[yshift=-4cm]
            \pic[
            braid/.cd,
            number of strands=3,
            strand 1/.style={black},
            strand 2/.style={black},
            strand 3/.style={black},
            gap=0.1,
            scale=1
            ] {braid={s_2^{1} s_1^{1}}};
        \end{scope}

        \begin{scope}[xshift=3cm]
            \pic[
            braid/.cd,
            number of strands=3,
            strand 1/.style={black},
            strand 2/.style={black},
            strand 3/.style={black},
            gap=0.1,
            scale=1
            ] {braid={s_1^{1} s_2^{1}}};
            \node at (1,-3) {\Huge $\vdots$};
            \begin{scope}[yshift=-4cm]
                \pic[
                braid/.cd,
                number of strands=3,
                strand 1/.style={black},
                strand 2/.style={black},
                strand 3/.style={black},
                gap=0.1,
                scale=1
                ] {braid={s_1^{1} s_2^{1}}};
            \end{scope}
        \end{scope}

        \begin{scope}[xshift=6cm]
            \pic[
            braid/.cd,
            number of strands=3,
            strand 1/.style={black},
            strand 2/.style={black},
            strand 3/.style={black},
            gap=0.1,
            scale=1
            ] {braid={s_2^{1} s_1^{1}}};
            \node at (1,-3) {\Huge $\vdots$};
            \begin{scope}[yshift=-4cm]
                \pic[
                braid/.cd,
                number of strands=3,
                strand 1/.style={black},
                strand 2/.style={black},
                strand 3/.style={black},
                gap=0.1,
                scale=1
                ] {braid={s_2^{1} s_1^{1}}};
            \end{scope}
        \end{scope}

        \begin{scope}[xshift=9cm]
            \pic[
            braid/.cd,
            number of strands=3,
            strand 1/.style={black},
            strand 2/.style={black},
            strand 3/.style={black},
            gap=0.1,
            scale=1
            ] {braid={s_1^{1} s_2^{1}}};
            \node at (1,-3) {\Huge $\vdots$};
            \begin{scope}[yshift=-4cm]
                \pic[
                braid/.cd,
                number of strands=3,
                strand 1/.style={black},
                strand 2/.style={black},
                strand 3/.style={black},
                gap=0.1,
                scale=1
                ] {braid={s_1^{1} s_2^{1}}};
            \end{scope}
        \end{scope}

        \foreach \i in {2,5,8} {
            \draw (\i,0) to[in=90,out=90] (\i+1,0);
        }

        \foreach \i in {2,5,8} {
            \draw (\i,-6.5) to[in=-90,out=-90] (\i+1,-6.5);
        }

        \draw (0,0) to[in=90,out=90] (11,0);
        \draw (0,-6.5) to[in=-90,out=-90] (11,-6.5);

        \draw (4,0) to[in=90,out=90] (10,0);

        \draw[white,line width=5pt] (1,0) to[in=90,out=90] (7,0);
        \draw (1,0) to[in=90,out=90] (7,0);

        \draw (4,-6.5) to[in=-90,out=-90] (10,-6.5);

        \draw[white,line width=5pt] (1,-6.5) to[in=-90,out=-90] (7,-6.5);
        \draw (1,-6.5) to[in=-90,out=-90] (7,-6.5);

    \end{tikzpicture}
}

\tikzset{
    threentoruso/.pic = {
        \foreach \i in {1,...,4} {
            \coordinate (A\i) at (0,\i);
        }

        \foreach \i in {1,...,4} {
            \coordinate (B\i) at ($(0,0)!1!-120:(A\i)$);
        }

        \foreach \i [evaluate=\i as \previ using int(\i - 1)] in {2,3,4} {
            \draw[->] (B\i) to[in=0,out=60] (A\previ);
        }

        \draw[white, line width=10pt] (B1) to[in=0,out=60] (A4);
        \draw[->] (B1) to[in=0,out=60] (A4);
    }
}

\tikzset{
    threetorusothree/.pic = {
        \foreach \i in {1,...,3} {
            \coordinate (A\i) at (0,\i);
        }

        \foreach \i in {1,...,3} {
            \coordinate (B\i) at ($(0,0)!1!-120:(A\i)$);
        }

        \foreach \i [evaluate=\i as \previ using int(\i - 1)] in {2,3} {
            \draw[->] (B\i) to[in=0,out=60] (A\previ);
        }

        \draw[white, line width=10pt] (B1) to[in=0,out=60] (A3);
        \draw[->] (B1) to[in=0,out=60] (A3);
    }
}

\tikzset{
    threentorus/.pic = {
        \foreach \i in {1,...,4} {
            \coordinate (A\i) at (0,\i);
        }

        \foreach \i in {1,...,4} {
            \coordinate (B\i) at ($(0,0)!1!-120:(A\i)$);
        }

        \foreach \i [evaluate=\i as \previ using int(\i - 1)] in {2,3,4} {
            \draw (B\i) to[in=0,out=60] (A\previ);
        }

        \draw[white, line width=10pt] (B1) to[in=0,out=60] (A4);
        \draw (B1) to[in=0,out=60] (A4);
    }
}

\tikzset{
    threetorusThree/.pic = {
        \foreach \i in {1,...,3} {
            \coordinate (A\i) at (0,\i);
        }

        \foreach \i in {1,...,3} {
            \coordinate (B\i) at ($(0,0)!1!-120:(A\i)$);
        }

        \foreach \i [evaluate=\i as \previ using int(\i - 1)] in {2,3} {
            \draw (B\i) to[in=0,out=60] (A\previ);
        }

        \draw[white, line width=10pt] (B1) to[in=0,out=60] (A3);
        \draw (B1) to[in=0,out=60] (A3);
    }
}

\tikzset{
    threentoruszero/.pic = {
        \foreach \i in {1,...,4} {
            \coordinate (A\i) at (0,\i);
        }

        \foreach \i in {1,...,4} {
            \coordinate (B\i) at ($(0,0)!1!-120:(A\i)$);
        }

        \foreach \i [evaluate=\i as \previ using int(\i - 1)] in {2,3} {
            \draw (B\i) to[in=0,out=60] (A\previ);
        }

        \draw[white, line width=10pt] (B1) to[in=0,out=60] (A3);
        \draw (B1) to[in=0,out=60] (A3);
        \draw (B4) to[in=0,out=60] (A4);
    }
}

\tikzset{
    threetoruszeroThree/.pic = {
        \foreach \i in {1,...,3} {
            \coordinate (A\i) at (0,\i);
        }

        \foreach \i in {1,...,3} {
            \coordinate (B\i) at ($(0,0)!1!-120:(A\i)$);
        }

        \foreach \i [evaluate=\i as \previ using int(\i - 1)] in {2} {
            \draw (B\i) to[in=0,out=60] (A\previ);
        }

        \draw[white, line width=10pt] (B1) to[in=0,out=60] (A2);
        \draw (B1) to[in=0,out=60] (A2);
        \draw (B3) to[in=0,out=60] (A3);

    }
}

\tikzset{
    threentorusinfty/.pic = {
        \foreach \i in {1,...,4} {
            \coordinate (A\i) at (0,\i);
        }

        \foreach \i in {1,...,4} {
            \coordinate (B\i) at ($(0,0)!1!-120:(A\i)$);
        }
 
        \coordinate (C3) at ($(0,0)!1!-30:(A3)$);
        \coordinate (C4) at ($(0,0)!1!-30:(A4)$);

        \foreach \i [evaluate=\i as \previ using int(\i - 1)] in {2,3} {
            \draw (B\i) to[in=0,out=60] (A\previ);
        }

        \draw[white, line width=10pt] (B1) to[in=0,out=60] (C3);
        \draw (B1) to[in=0,out=60] (C3);
        \draw (C3) to[in=150,out=150] (C4);
        \draw (A3) to[in=0,out=0] (A4);
        \draw (B4) to[in=-30,out=60] (C4);
    }
}

\tikzset{
    threetorusinftyThree/.pic = {
        \foreach \i in {1,...,3} {
            \coordinate (A\i) at (0,\i);
        }

        \foreach \i in {1,...,3} {
            \coordinate (B\i) at ($(0,0)!1!-120:(A\i)$);
        }

        \coordinate (C2) at ($(0,0)!1!-30:(A2)$);
        \coordinate (C3) at ($(0,0)!1!-30:(A3)$);

        \foreach \i [evaluate=\i as \previ using int(\i - 1)] in {2} {
            \draw (B\i) to[in=0,out=60] (A\previ);
        }

        \draw[white, line width=10pt] (B1) to[in=0,out=60] (C2);
        \draw (B1) to[in=0,out=60] (C2);
        \draw (C2) to[in=150,out=150] (C3);
        \draw (A2) to[in=0,out=0] (A3);
        \draw (B3) to[in=-30,out=60] (C3);
    }
}

\tikzset{
    threetorustwo/.pic = {
        \foreach \i in {1,...,2} {
            \coordinate (A\i) at (0,\i);
        }

        \foreach \i in {1,...,2} {
            \coordinate (B\i) at ($(0,0)!1!-120:(A\i)$);
        }

        \foreach \i [evaluate=\i as \previ using int(\i - 1)] in {2} {
            \draw (B\i) to[in=0,out=60] (A\previ);
        }

        \draw[white, line width=10pt] (B1) to[in=0,out=60] (A2);
        \draw (B1) to[in=0,out=60] (A2);
    }
}

\tikzset{
    threetoruszerotwo/.pic = {
        \foreach \i in {1,...,2} {
            \coordinate (A\i) at (0,\i);
        }

        \foreach \i in {1,...,2} {
            \coordinate (B\i) at ($(0,0)!1!-120:(A\i)$);
        }

        \draw (B1) to[in=0,out=60] (A1);
        \draw (B2) to[in=0,out=60] (A2);
    }
}

\tikzset{
    threetorusinftytwo/.pic = {
        \foreach \i in {1,...,2} {
            \coordinate (A\i) at (0,\i);
        }

        \foreach \i in {1,...,2} {
            \coordinate (B\i) at ($(0,0)!1!-120:(A\i)$);
        }

        \coordinate (C1) at ($(0,0)!1!-30:(A1)$);
        \coordinate (C2) at ($(0,0)!1!-30:(A2)$);

        \draw[white, line width=10pt] (B1) to[in=0,out=60] (C1);
        \draw (B1) to[in=0,out=60] (C1);
        \draw (C1) to[in=150,out=150] (C2);
        \draw (A1) to[in=0,out=0] (A2);
        \draw (B2) to[in=-30,out=60] (C2);
    }
}

\newcommand{\syll}{
    \tdplotsetmaincoords{70}{0}
    \begin{tikzpicture}[
        tdplot_main_coords,
        scale=1.5,
        line cap=round,
        line join=round]

        \def\R{1.0}        % 半径
        \def\H{3.0}        % 高さ
        \def\Turns{1}   % 回転数
        \pgfmathsetmacro{\MaxAngle}{360 * \Turns}
        
        \def\FixedSideAngle{0} 

        \draw[red!70, dashed, thick] plot[domain=0:360, samples=60] ({\R*cos(\x)}, {\R*sin(\x)}, 0);
        
        \draw[red!70, dashed, thick] plot[domain=0:360, samples=60] ({\R*cos(\x)}, {\R*sin(\x)}, \H);
        
        \draw[red!70, dashed, thick] ({\FixedSideAngle}:\R) -- ++(0,0,\H);
        \draw[red!70, dashed, thick] ({\FixedSideAngle+180}:\R) -- ++(0,0,\H);

        \def\strands{5}
        \foreach \k in {1,...,\strands} {
            \def\offset{\k * 360 / \strands}
            
            \draw[black, ultra thick, dashed] 
                plot[domain=0:\MaxAngle, samples=100, smooth, variable=\t] 
                ({\R*cos(\t + \offset)}, {\R*sin(\t + \offset)}, {\H * \t / \MaxAngle});

            \fill[blue!80!black] ({\R*cos(\offset)}, {\R*sin(\offset)}, 0) circle (1.5pt);
            \fill[blue!80!black] ({\R*cos(\MaxAngle + \offset)}, {\R*sin(\MaxAngle + \offset)}, \H) circle (1.5pt);
        }

        \def\strands{5}
        \foreach \k in {1,...,\strands} {
            \pgfmathsetmacro{\offset}{\k * 360 / \strands}
            
            \pgfmathsetmacro{\numCycles}{ceil(\MaxAngle / 360)}
            
            \foreach \n in {0,1,...,\numCycles} {
                \pgfmathsetmacro{\tStart}{max(180 - \offset + \n*360, 0)}
                \pgfmathsetmacro{\tEnd}{min(360 - \offset + \n*360, \MaxAngle)}
                
                \ifdim\tStart pt<\MaxAngle pt
                    \draw[black, ultra thick] 
                        plot[domain=\tStart:\tEnd, samples=50, smooth, variable=\t] 
                        ({\R*cos(\t + \offset)}, {\R*sin(\t + \offset)}, {\H * \t / \MaxAngle});
                \fi
            }
    }
    \end{tikzpicture}
}